\documentclass[11pt]{amsart}
\usepackage{amssymb,amsfonts,amsmath}
\usepackage[all]{xy}
\usepackage[shortlabels]{enumitem}
\usepackage{amssymb}
\usepackage{lipsum}

\usepackage{hyperref,color}

\newtheorem{theorem}{Theorem}[section]
\newtheorem{proposition}{Proposition}[section]

\newtheorem{corollary}{Corollary}[section]

\theoremstyle{definition}

\newtheorem{lemma}{Lemma}[section]

\theoremstyle{remark}

\theoremstyle{theoremA}
\newtheorem*{theoremA}{Theorem A}

\newcommand{\Div}{\mbox{div}_M}
\newcommand{\real}{\mathbb{R}}

\newcommand{\si}{\sigma}

\newcommand{\la}{\lambda}
\newcommand{\al}{\alpha}

\newcommand{\na}{\nabla}
\newcommand{\ep}{\epsilon}
\newcommand{\ga}{\gamma}

\newcommand{\be}{\beta}
\newcommand{\De}{\Delta}

\newcommand{\lan}{\langle}
\newcommand{\ran}{\rangle}
\newcommand{\p}{\partial}

\newcommand{\dv}{\mathrm{div}}
\newcommand{\hs}{\mathrm{Hess}}
\newcommand{\tr}{\mathrm{Tr\,}}

\newcommand{\rad}{\mathrm{rad}}

\newcommand{\II}{{I\!I}}
\newcommand{\spp}{\mathrm{supp}\,}

\title[Hardy and Rellich Inequalities on Submanifolds]{Hardy and Rellich Inequalities for submanifolds in Hadamard spaces}
\author{M. Batista}
\address{Instituto de Matem\'atica, Universidade Fe\-deral de Alagoas, Macei\'o, AL, CEP 57072-970, Brazil}\email{mhbs@mat.ufal.br}
\author{H. Mirandola}
\address{Instituto de Matem\'atica, Universidade Fe\-deral do Rio de Janeiro, Rio de Janeiro, RJ, CEP 21945-970, Brasil} \email{mirandola@ufrj.br}
\author{F. Vit\'orio}
\address{Instituto de Matem\'atica, Universidade Fe\-deral de Alagoas, Macei\'o, AL, CEP 57072-970, Brazil}\email{feliciano@pos.mat.ufal.br}

\begin{document}
\maketitle
\begin{abstract} Some of the most known integral inequalities are the Sobolev, Hardy and Rellich inequalities in Euclidean spaces. In the context of submanifolds, the Sobolev inequality was proved by Michael-Simon \cite{MS} and Hoffman-Spruck \cite{HS}. 
Since then, a sort of applications to the submanifold theory has been derived from those inequalities. Years later, Carron \cite{C} obtained a Hardy inequality for submanifolds in Hadamard spaces. In this paper, we prove the general Hardy and Rellich Inequalities for submanifolds in Hadamard spaces. Some applications are given and we also analyse the equality cases.
\end{abstract}


\section{Introduction}

Over the years, geometers have been interested in understanding how integral inequalities imply geometric and topological obstructions on Riemannian manifolds. Under  this purpose,  some integral inequalities lead us to study positive solutions to critical singular quasilinear elliptic problems, sharp constants, existence, non-existence, rigidity and symmetry results for extremal functions on subsets in the Euclidean space. About these subjects, a comprehensive material can be found, for instance, in \cite{BT}, \cite{BN}, \cite{C}, \cite{CC}, \cite{CKN}, \cite{CW}, \cite{HS}, \cite{MS}, \cite{KO} and references therein.

In the literature, some of the most known integral inequalities are the Sobolev, Hardy and Rellich inequalities. The validity of these inequalities and their corresponding sharp constants on a given Riemannian manifold measures, in some suitable sense, how close that manifold is to an Euclidean space. These inequalities can be applied to obtain results such as comparison for the  volume growth, estimates of the essencial spectrum for the Schr\"{o}dinger operators, parabolicity,  among others properties (see, for instance, \cite{L, dCX, X}). 

In this paper, we obtain the general Hardy and Rellich Inequalities for submanifolds in Hadamard ambient spaces using an elementary and very efficient approach. Furthermore, we analyse the equalities cases. At the end of this paper, we give some applications. 

This paper is organized as follows. In the section  below, we introduce our notation and show some simple facts about isometric immersions that will be useful in the proofs of the present paper. In Section \ref{cptHardy} and \ref{cpt-Rellich}, we prove Hardy and Rellich Inequalities for submanifolds in Hadamard spaces. In Section \ref{sec-eq-blw} we study the equality cases and do some blowup analysis to show that all the integral in our inequalities converge. The later section is devoted to applications.

\section{Preliminaries}

Let us start recalling some basic concepts, notations and properties about submanifolds. First, let $M = M^k$ be a $k$-dimensional Riemannian manifold with (possibly nonempty) boundary $\p M$. Assume $M$ is isometrically immersed in a $n$-dimensional Hadamard space $\bar M^n$ (that means $\bar M$ is a complete simply-connected manifold with non positive sectional curvature). We will denote by $f:M\to \bar M$ the isometric immersion. By abuse of notation, sometimes, we will identify $f(x)=x$, for all $x\in M$. No restriction on the codimension of $f$ is required. Let $\lan\cdot,\cdot\ran$ denote the Riemannian metric on $\bar M$ and consider the same notation to the metric induced on $M$. Associated to these metrics, consider $\bar\na$ and $\na$ the Levi-Civita connections on $\bar M$ and $M$, respectively. It is well known that $\nabla_Y Z = (\bar \na_YZ)^\top$, where $\top$ means the orthogonal projection onto the tangent  bundle $TM$. The Gauss equation says $$\bar \na_Y Z = \na_Y Z + \II(Y,Z),$$
where $\II$ is a quadratic form named by  {\it second fundamental form}. The mean curvature vector is defined by $H=\tr_M \,\II$. 

Fixed $\xi\in \bar M$,  let $r_{\xi}=d_{\bar M}(\cdot\,,\xi)$ be the distance on $\bar M$ from $\xi$. Since $\bar M$ is a Hadamard space, by the Hessian comparison theorem (see Theorem 2.3 page 29 of \cite{PRS}), we have 
\begin{equation}\label{hes-comp}
\hs_{\,r}(v,v) \ge \frac{1}{r}(1-\lan \bar\na r,v\ran^2),
\end{equation} 
for all points in $\bar M^* = \bar M\setminus \{\xi\}$ and vector fields $v:\bar M^*\to T\bar M$ with $|v|=1$. Here, $\bar \na r$ denotes the gradient vector, on $\bar M$, of $r$.
For a vector field $Y:M\to T\bar M$, the  divergence of $Y$ on $M$ is given by
$$\Div Y = \sum_{i=1}^k \lan \bar\na_{e_i}Y, e_i\rangle,$$
where $\{e_1,\cdots, e_k\}$ denotes a local orthonormal frame on $M$.  By simple computations, one has 

\begin{lemma}\label{prop} Let $Y:M\to T\bar M$ be a vector field and $\psi \in C^1(M)$. The following items hold
\begin{enumerate}[(a)]
\item\label{prop-a} $\Div Y = \Div Y^\top - \langle H, Y\rangle;$
\medskip
\item\label{prop-b} $\Div (\psi Y) = \psi~\Div Y + \lan \na^M\psi, Y\rangle$,
\end{enumerate}
where $\na^M \psi$ denotes the gradient vector, on $M$, of $\psi$.
\end{lemma}

Henceforth, we will denote by $X:M\to T\bar M$ the radial vector $X = r\bar \na r$. We see that $X$ is continuous, with $|X|=r$ and differentiable in $M\setminus \{\xi\}$. The following lemma holds

\begin{lemma}\label{formula}
Let $\ga$ and $p$ be real numbers and $\psi\in C^1(M)$. The isometric immersion $f:M\to \bar M$ satisfies: 
\begin{equation}\label{deriv_X}
\dv_M(\frac{\psi^p X^\top}{|X|^\ga}) \ge \frac{\psi^p}{|X|^\ga}[k-\ga+\frac{\ga |X^\perp|^2}{|X|^2}]  + \lan \, p \na^M \psi + \psi H, \frac{\psi^{p-1} X}{|X|^\ga}\ran
\end{equation}
and
\begin{equation}\label{lapl_X}
\frac{1}{2-\ga}\Delta(\frac{1}{|X|^{\ga-2}}) \ge \frac{k-\ga}{|X|^\ga} + \ga~\frac{|X^\perp|^2}{|X|^{\ga+2}} + \lan\frac{X}{|X|^\ga}, H\ran, 
\end{equation}
where $(\cdot)^\perp$ means the orthogonal projection on the normal bundle of $M$. Moreover, for both inequalities, (\ref{deriv_X}) and (\ref{lapl_X}), their corresponding equalities occur if and only if the radial curvature $(K_{\rad})_{\xi}=0$ in $\spp(\psi)$ and $M$, respectively.  
\end{lemma}
Before to prove Lemma \ref{formula}, let us recall the definition of radial sectional curvature. Let $x\in \bar M$ and, since $\bar M$ is complete, let $\ga:[0,t_0=r_{\xi}(x)]\to \bar M$ be a minimizing geodesic in $\bar M$ from $\xi$ to $x$. For all  orthonormal pair of vectors $Y,Z\in T_x\bar M$ we define $(\bar K_\rad)_{\xi}(Y,Z)=\lan \bar R(Y,\ga'(t_0))\ga'(t_0), Z\ran$. 

\begin{proof}[Proof of Lemma \ref{formula}] By the Hessian comparison theorem (\ref{hes-comp}), $\dv_M X \ge k$ and $\dv_M X=k$ if and only if $(K_{\rad})_\xi=0$ in $M$.  Furthermore, since $|X|=r$, the gradient vector $\na^M |X| = \frac{X^T}{|X|}$. By Lemma \ref{prop} \ref{prop-b}, and using that $|X|^2 = |X^\top|^2+|X^\perp|^2$, one has
\begin{equation*}
\Div (\frac{X}{|X|^\ga}) = \frac{1}{|X|^\ga}\Div X +\lan \na^M (\frac{1}{|X|^\ga}),~X\ran 
\ge \frac{k-\ga}{|X|^\ga} + \ga~\dfrac{|X^\perp|^2}{|X|^{\ga+2}}.
\end{equation*}
Thus, by Lemma \ref{prop} \ref{prop-a}, one has
\begin{equation*}
\Div(\frac{X^T}{|X|^\ga}) = \Div(\frac{X}{|X|^\ga}) + \lan  H,\frac{X}{|X|^\ga}\ran \ge  \frac{k-\ga}{|X|^\ga} + \ga~\frac{|X^\perp|^2}{|X|^{\ga+2}} + \lan\frac{X}{|X|^\ga}, H\ran.
\end{equation*}
Hence, using Lemma \ref{prop} \ref{prop-b}, Item (\ref{deriv_X}) follows.

Now, since $\na^M(\frac{1}{|X|^{\ga-2}}) = -(\ga-2)\,|X|^{-(\ga-2)-1}\, \na^M(|X|) = -(\ga-2)\,\frac{X^T}{|X|^{\ga}}$, one has $\De(\frac{1}{|X|^{\ga-2}}) = -(\ga-2)\, \Div(\frac{X^T}{|X|^\ga})$. Item (\ref{lapl_X}) follows. Lemma \ref{formula} is proved.
\end{proof}
\section{Hardy Inequalities for submanifolds}\label{cptHardy}

Carron \cite{C} proved the following result:
\begin{theoremA}[Carron]\label{carron-teo} Let $M^k$ be a complete non-compact  manifold, with $k\ge 3$, isometrically immersed in a Hadamard space $\bar M$. Let $r = d_{\bar M}(\cdot\,,\xi)$ be the distance on $\bar M$ from any fixed point $\xi\in \bar M$. Then, for all compactly supported function $\psi\in C_0^1(M)$, the following Hardy inequality holds:
\begin{equation*} \frac{(k-2)^2}{4} \int_M \frac{\psi^2}{r^2} \le \int_M \big[|\na^M \psi|^2 + \frac{(k-2){|H|}\psi^2}{2r}\big].
\end{equation*}
\end{theoremA}

For the special case that $M$ is a minimal submanifold in Hadamard manifold $\bar M$, Theorem A was generalized by Bianchini, Mari and Rigoli \cite{bmr} in order to obtain some Yamabe type equations.

It is worthwhile to recall that in the present paper we are assuming $M$ compact with (possibly nonempty) boundary $\p M$. Comparing with assumptions in Theorem A, if $M$ were assumed complete and non-compact then, for a given $\psi\in C^1_0(M)$, we could consider a compact manifold $M'$ with smooth boundary $\p M'$ satisfying $\spp(\psi)\subset M'\subset M$.

Our first result is a Hardy-type inequality for submanifolds that generalizes of Theorem A. As a consequence of Theorem \ref{teo-hardy-norm}, it follows that the equality in Carron's Theorem occurs if and only if $\psi=0$ in $M$.

\begin{theorem}\label{teo-hardy}
Let $M^k$ be a compact manifold with (possibly nonempty) boundary $\p M$. Assume $M$ is isometrically immersed in a Hadamard manifold $\bar M$ and let $r = d_{\bar M}(\cdot\,,\xi)$ be the distance in $\bar M$ from any fixed point $\xi$. Let $1\le p < \infty$ and $-\infty<\ga<k$. For all $0\le \psi\in C^1(M)$, it holds
\begin{align}
\frac{(k-\ga)^p}{p^p}\int_M \frac{\psi^p}{r^\ga} &+ \frac{(k-\ga)^{p-1}}{p^{p-1}}\int_M \frac{\psi^p}{r^\ga}\Phi^+ 
\le \int_M \frac{ |\na^M \psi|^p}{r^{\ga-p}} \nonumber\\& + \frac{(k-\ga)^{p-1}}{p^{p-1}}\int_M \frac{\psi^p}{r^\ga}\Phi^- + \frac{(k-\ga)^{p-1}}{p^{p-1}}\,\int_{\p M}\frac{\psi^p}{r^{\ga-1}}\lan \bar\na r,\nu\ran, \label{des-hardy}
\end{align}
where $\Phi = \ga|\bar \na r^\perp|^2 + r\lan \bar\na r,H\ran$. Here, $\Phi^+,\Phi^-$ denote the positive and negative parts of $\Phi$, respectively, and $\nu$ is the exterior conormal to $\p M$. Moreover, if $p>1$, the equality above holds if and only if $\psi=0$ in $M$.
\end{theorem}

To see how Theorem \ref{teo-hardy} generalizes Theorem A above, if $\ga\ge 0$ then 
\begin{align*}
&\Phi^+  = \ga|\bar \na r^\perp|^2 + r\lan \bar\na r,H\ran^+  \ge \ga|\bar \na r^\perp|^2 \\
&\Phi^-  = r\lan \bar\na r,H\ran^- \le r|\lan \bar\na r,H\ran|.
\end{align*}
For $p=\ga=2$, it holds $\int_M \frac{\psi^2}{r^2}\Phi^- \le \int_M \frac{\psi^2}{r}|H|$, hence Theorem A follows as Theorem \ref{teo-hardy}. Furthermore, as a consequence of Theorem \ref{teo-hardy}, it holds:

\begin{corollary} \label{teo-hardy-carron}
Let $M^k$ be a compact manifold with (possibly nonempty) boundary $\p M$. Assume $M$ is isometrically immersed in a Hadamard manifold $\bar M$ and let $r = d_{\bar M}(\cdot\,,\xi)$ be the distance in $\bar M$ from a fixed point $\xi\in \bar M$. Let $1\le p < \infty$ and $-\infty<\ga<k$. For all $0\le \psi \in C^1(M)$, it holds
\begin{eqnarray*}
\frac{(k-\ga)^p}{p^p} \int_M \frac{\psi^p}{r^\ga} + \frac{\ga (k-\ga)^{p-1}}{p^{p-1}} \int_M \frac{\psi^p}{r^{\ga}}|\bar\na r^\perp|^2 &\le& \int_M \frac{ |\na^M \psi|^p}{r^{\ga-p}}\nonumber\\
&& \hspace{-7cm} \ +\ \frac{(k-\ga)^{p-1}}{p^{p-1}} \int_M \frac{\psi^p}{r^{\ga-1}} |\lan H,\bar\na r\ran| \ +\ \frac{(k-\ga)^{p-1}}{p^{p-1}}\,\int_{\p M}\frac{\psi^p}{r^{\ga-1}}\lan \bar\na r,\nu\ran.
\end{eqnarray*}
Moreover, if $p>1$, the equality above holds if and only if $\psi=0$ in $M$.
\end{corollary}
\begin{proof} If $\ga\ge 0$, Corollary \ref{teo-hardy-carron} follows from Theorem \ref{teo-hardy} togheter with the facts $\Phi^+\ge \ga|\bar\na r^\perp|$ and $\Phi^- \le r|\lan\bar\na r,H\ran|$. Now, if $\ga<0$ then $\int_M \frac{\psi^p}{r^\ga}$ exists. Thus, since $\Phi = \Phi^+-\Phi^- = \ga|\bar \na r^\perp|^2 + r\lan \bar\na r,H\ran$, by Theorem \ref{teo-hardy},
\begin{align*}
\frac{(k-\ga)^p}{p^p}\int_M \frac{\psi^p}{r^\ga} &+ \frac{(k-\ga)^{p-1}}{p^{p-1}}\int_M \frac{\psi^p}{r^\ga}\big[\ga|\bar \na r^\perp|^2 + r\lan \bar\na r,H\ran\big] 
\nonumber\\& \le \int_M \frac{ |\na^M \psi|^p}{r^{\ga-p}}  +  \frac{(k-\ga)^{p-1}}{p^{p-1}}\,\int_{\p M}\frac{\psi^p}{r^{\ga-1}}\lan \bar\na r,\nu\ran,
\end{align*}
Corollary \ref{teo-hardy-carron} follows from the triangular inequality, $-r\lan \bar\na r,H\ran\le r|\lan \bar\na r,H\ran|$.
\end{proof}

Now, we will prove Theorem \ref{teo-hardy}.
\begin{proof}[Proof of Theorem \ref{teo-hardy}] For convenience, we will prove this theorem by considering the radial vector field $X=r\bar\na r$. Set $0\le \psi \in C^1(M)$. 

First,  assume $p>1$.  Write $\ga=\al+\be+1$ with $\al,\be\in \real$. By Lemma \ref{formula},
\begin{align}
\dv_M(\frac{\psi^p X^\top}{|X|^\ga})& \ge \frac{\psi^p}{|X|^\ga}[k-\ga+\frac{\ga |X^\perp|^2}{|X|^2}]  + \lan p \na^M \psi + \psi H, \frac{\psi^{p-1} X}{|X|^\ga}\ran \nonumber
\\&= \frac{\psi^p}{|X|^\ga}[k-\ga + \frac{\ga |X^\perp|^2}{|X|^2}]  + \lan \frac{p \na^M \psi}{|X|^\al}, \frac{\psi^{p-1} X}{|X|^{\be+1}} \ran + \frac{\psi^p}{|X|^\ga}\lan H,X\ran.\label{cauchy-young-2}
\end{align}
Moreover, if the equality in (\ref{cauchy-young-2}) holds then the radial curvature $(K_{\rad})_{\xi}=0$ in $\spp\psi$. Thus, using Cauchy-Schwarz Inequality and Young Inequality with $\ep>0$, 
\begin{eqnarray}
\dv_M(\frac{\psi^p X^\top}{|X|^\ga}) &\ge& 
\frac{\psi^p}{|X|^\ga}\big[k-\ga + \frac{\ga |X^\perp|^2}{|X|^2} + \lan \vec{H},X\ran\big]\nonumber\\&&  - \frac{1}{p\ep^p}\frac{ |p\na^M \psi|^p}{|X|^{p\al}} - \frac{\ep^q}{q} \frac{\psi^{(p-1)q}}{|X|^{q\be}},\label{young}
\end{eqnarray}
where $p>1$ and $q=\frac{p}{p-1}$. 

Now, we take $\be$ satisfying $\be q = \ga$, that is, $\be  = (p-1)(\al+1)$. Since $\ga=\al+\be+1$, we see easily that $p\al = \ga-p$. Thus,
\begin{eqnarray}
h(\ep) \frac{\psi^p}{|X|^\ga} + \ep^p\frac{\psi^p}{|X|^\ga}\big[\ga \frac{|X^\perp|^2}{|X|^2} + \lan H,X\ran \big] \nonumber\\ &&\hspace{-2cm} \ \le \ p^{p-1} \frac{ |\na^M \psi|^p}{|X|^{\ga-p}}  + \ep^p\, \dv_M (\frac{\psi^p X^\top}{|X|^\ga}),\label{antes_max}
\end{eqnarray}
where $h(\ep)=\ep^p(k-\ga - \frac{\ep^q}{q})$. The function $h$ achieves its maximum at the instant $\ep = (k-\ga)^{\frac{p-1}{p}}$, with $h(\ep) = \frac{(k-\ga)^p}{p}$. Thus, it holds
\begin{equation}
\frac{(k-\ga)^p}{p^p} \frac{\psi^p}{|X|^\ga} + \frac{(k-\ga)^{p-1}}{p^{p-1}}\frac{ \psi^p\Phi}{|X|^\ga}\leq\ \frac{ |\na^M \psi|^p}{|X|^{\ga-p}} \ +\ \frac{(k-\ga)^{p-1}}{p^{p-1}}\,\dv_M \frac{\psi^p X^\top}{|X|^\ga},\label{before}
\end{equation}
where $\Phi = \ga \frac{|X^\perp|^2}{|X|^2} + \lan H,X\ran = \ga|\bar \na r^\perp|^2 + r\lan \bar\na r,H\ran $. We observe that (\ref{before}) remains valid for $p=1$, just considering (\ref{before}) with $p>1$ and taking $p\to 1$.

We claim that $\int_M \dv_M(\frac{\psi^p X^\top}{|X|^\ga}) = \int_{\p M} \frac{\psi^p}{|X|^\ga}\lan X,\nu \ran$ holds even when $0\in M$. In fact, if $0\notin M$, we can apply the divergence theorem directly. If $0\in M$, we consider a small regular value $r_0>0$ for the distance function $r=|X|$. We have 
\begin{eqnarray}
\int_{M\cap [|X|>r_0]}\dv_M(\frac{\psi^p X^\top}{|X|^\ga}) &=& \int_{[\p M \cap [|X|>r_0]]\cup [M\cap [|X|=r_0]]} \frac{\psi^p}{|X|^\ga}\lan X,\nu\ran\nonumber\\
&=&\int_{[\p M \cap [|X|>r_0]]} \frac{\psi^p}{|X|^\ga}\lan X,\nu\ran + O(r_0^{1-\ga})O(r_0^{k-1}), \label{stokes}
\end{eqnarray}
as $r_0\to 0$. Thus, since $k-\ga>0$ e $\int_{\p M} \frac{\psi^p}{|X|^{\ga-1}}$ exists (see (\ref{conv-int-bry}) in Section \ref{sec-eq-blw}), by the dominated convergence theorem, taking $r_0\to 0$, our claim follows. 

Using (\ref{before}) and (\ref{stokes}), 
\begin{eqnarray}
\frac{(k-\ga)^p}{p^p} \int_M \frac{\psi^p}{|X|^\ga} + \frac{(k-\ga)^{p-1}}{p^{p-1}} \int_M \frac{\psi^p}{|X|^\ga}\Phi^+ &\le& \int_M \frac{|\na^M \psi|^p}{|X|^{\ga-p}}\nonumber\\
&& \hspace{-7cm} \ +\ \frac{(k-\ga)^{p-1}}{p^{p-1}} \int_M \frac{\psi^p}{|X|^\ga}\Phi^- \ +\ \frac{(k-\ga)^{p-1}}{p^{p-1}}\,\int_{\p M}\frac{\psi^p}{|X|^\ga}\lan X,\nu\ran,\label{after}
\end{eqnarray}
where $\Phi^+$ and $\Phi^-$ denotes the positive and negative parts of $\Phi$ and $\nu$ is the outward pointing unit normal field of the boundary $\p M$. 

Now, assume the equality in (\ref{after}) holds. We will prove that $\psi=0$ in M. In fact, the equality in (\ref{after}) implies the radial seccional curvature $(\bar K_{\rad})_\xi = 0$ in $\spp\psi$ and also imply the equalities in the Cauchy and Young Inequalities, with $\ep^q=k-\ga$, holds in (\ref{young}). Thus, we obtain  
\begin{align}
\na^M \psi &=  -\lambda X, \mbox{ and } \label{eq-cauchy}\\
\frac{p^p |\na^M \psi|^p}{\ep^p |X|^{\al p}} &= \frac{\ep^q \psi^{(p-1)q}}{|X|^{\be q}},\label{eq-young}
\end{align}
where $\lambda\ge 0$ is a continuous function on $M$, and the constants $\al$ and $\be$ satisfy $\al p = \ga - p$ and $\be q = \ga$. Since $\ep^{p+q}=\ep^{pq}=(k-\ga)^p$, by (\ref{eq-young}),
\begin{equation}\label{div0}
\frac{(k-\ga)^p}{p^p}\frac{\psi^p}{|X|^\ga} = \frac{|\na^M \psi|^p}{|X|^{\ga-p}}.
\end{equation}
Hence, from (\ref{eq-cauchy}), 
\begin{equation}\label{radial}
\na^M\psi = -\frac{k-\ga}{p}\psi X.    
\end{equation}
In particular, it holds $X^T=X$ and $X = \frac{1}{2}\na^M(|X|^2)$. By contradiction, assume that the open subset $W = \{\psi>0\}$ is nonempty. Using (\ref{radial}) and the fact $X = \frac{1}{2}\na^M(|X|^2)$, one has  $\na^M(\log\psi + \frac{k-\ga}{2p}|X|^2)=0$, in $W$. Hence, 
\begin{equation}\label{expo}
\psi = D e^{-\frac{k-\ga}{2p}|X|^2},
\end{equation}
in $W$, for some constant $D>0$. This implies $\psi>0$ everywhere in $M$. Moreover, using that $X^\perp=0$, by (\ref{before}) (with the equality case) and (\ref{div0}), one also obtains 
\begin{equation}\label{conserv}
\dv_M(\frac{\psi^p X^T}{|X|^\ga})=0,
\end{equation}
everywhere in $M$. On the other hand, since $\psi^p = D^p e^{-\frac{k-\ga}{2}|X|^2}$, and again using $X^T=X$, it follows by (\ref{deriv_X}) in Lemma \ref{formula} (with equality since the radial curvature $(K_{\rad})_{\xi}=0$ in $M$), and (\ref{radial}), we obtain
\begin{align*}
\dv_M(\frac{\psi^p X^T}{|X|^\ga}) &= \frac{1}{|X|^\ga}[(k-\ga)\psi^p + p\psi^{p-1}\lan \na^M\psi,X\ran]\\
&= \frac{(k-\ga)\psi^p}{|X|^\ga}[1-|X|^2].
\end{align*}
Thus, by (\ref{conserv}), one has $\psi^p(1-|X|^2)=0$ everywhere in $M$, which is a contradiction, since $|X|^2=1$ would imply $X^T=0$ in $M$, hence $X=0$ everywhere in $M$. Therefore, $\psi=0$ in $M$. Theorem \ref{teo-hardy} is proved.
\end{proof}

It is straightforward to verify that with the same steps as in the proof of Theorem \ref{teo-hardy-carron} above, but applying in (\ref{cauchy-young-2}) Cauchy-Schwarz and Young Inequalities directly to $\lan p \na^M \psi + \psi H, \frac{\psi^{p-1} X}{|X|^\ga}\ran$ instead $\lan p \na^M \psi, \frac{\psi^{p-1} X}{|X|^\ga}\ran$, one can prove the following variance of Theorem \ref{teo-hardy-carron}.

\begin{theorem} \label{teo-hardy-norm} Under the same hypothesis of Theorem \ref{teo-hardy-carron}, for all $0\le \psi \in C^1(M)$, it holds
\begin{align*}
\frac{(k-\ga)^p}{p^p} \int_{M} \frac{\psi^p}{r^{\ga}} &+ \frac{\ga(k-\ga)^{p-1}}{p^{p-1}}\int_M \frac{\psi^p}{r^{\ga}}|\bar\na r^\perp|^2
\\&\le  \int_M \frac{1}{r^{\ga-p}}\big|\na^M \psi+\frac{\psi H}{p}\big|^p +  \frac{(k-\ga)^{p-1}}{p^{p-1}} \int_{\p M} \frac{\psi^p}{r^{\ga-1}}.
\end{align*}
Moreover, if $p>1$, the equality holds if and only if $\psi=0$ in $M$.
\end{theorem}
\section{Rellich Inequalities for submanifolds}\label{cpt-Rellich}

\begin{theorem}\label{rellich} Let $M^k$ be a compact manifold with (possibly nonempty) boundary $\p M$. Assume $M$ is isometrically immersed in a Hadamard manifold $\bar M$. Let $r=d_{\bar M}(\cdot\,,\xi)$ be the distance in $\bar M$ from a fixed point $\xi\in \bar M$. Let $p\ge 1$ and $2<\ga<k$. For the special case that $\bar M$ has  radial curvature $(\bar K_{\rad})_{\xi}=0$ (e.g., $\bar M=\real^n$), we relax the hypothesis about $\ga$ by assuming $2-(p-1)(k-2)<\ga<k$. For all $0\le \psi \in C^1(M)$, it holds:
\begin{equation}\label{ineq-rellich}
A\int_M \frac{\psi^p}{r^\ga} + B\int_M \frac{\psi^p}{r^\ga}\Phi^+ \le \int_M \frac{|\De\psi|^p}{r^{\ga-2p}} + B\int_M \frac{\psi^p}{r^{\ga}}\Phi^- + \int_{\p M} \frac{\psi^{p-1}}{r^{\ga-2}}\lan W,\nu\ran.
\end{equation}
Here $\Phi = \ga|\bar\na r^\perp|^2 + r\lan \bar\na r,H \ran$, and 
\begin{align*}
A &= A(k,\ga,p) = \frac{E^p}{p^p} =  \frac{(k-\ga)^p}{p^{2p}}\big[\ga-2+(p-1)(k-2)\big]^p>0,\\
B &= B(k,\ga,p)  =  \frac{E^{p-1}}{p^{p-1}}\big[\frac{E}{k-\ga} + \frac{p-1}{p}\big]>0, \ \ \mbox{ and }\\
W &= W(k,\ga,p)  = \frac{ E^{p-1}}{p^{p-2}}\na^M\psi + B\frac{\psi\bar\na r^T}{r},
\end{align*}
with 
\begin{equation*}
E = \frac{k-\ga}{p}\big[\ga-2+(p-1)(k-2)\big]>0.
\end{equation*}
Furthermore, if $p>1$, the equality in (\ref{ineq-rellich}) holds if and only if $\psi=0$ in $M$.
\end{theorem}

A special case occurs when $\ga=2p$. In this case, it holds
\begin{corollary}\label{cor_rellich} Let $M$ be a compact manifold with (possibly nonempty) boundary $\p M$. Assume $M$ is isometrically immersed in a Hadamard manifold $\bar M$ and let $r = d_{\bar M}(\cdot\,,\xi)$ be the distance in $\bar M$ from a fixed point $\xi$. Assume $k>2p$, for some $p>1$. Then, for all $0\le \psi \in C^1(M)$, it holds
\begin{equation*}
A\int_M \frac{\psi^p}{r^{2p}} + B\int_M \frac{\psi^p}{r^{2p}}\Phi^+ \le \int_M |\De\psi|^p + B\int_M \frac{\psi^p}{r^{2p}}\Phi^- + \int_{\p M}\frac{\psi^{p-1}}{r^{2(p-1)}}\lan W,\nu\ran.
\end{equation*}
Here, $\Phi = 2p|\bar\na r^\perp|^2+ r\lan \bar\na r,H \ran$, and
\begin{align*}
A &= \frac{k^p}{p^{2p}}(k-2p)^p(p-1)^p;\\
B &= \frac{k^{p-1}}{p^{2p-1}}(k-2p)^{p-1}(p-1)^{p}(k+1); \ \  \mbox{and} \\
W &= \frac{(k-2p)^{p-1}k^{p-1}(p-1)^{p-1}}{p^{2p-3}}\big[\na^M\psi + \frac{(p-1)(k+1)}{p^2}\frac{\psi\bar\na r^\top}{r}\big]. 
\end{align*}
Moreover, the equality holds if and only if $\psi=0$ everywhere in $M$.
\end{corollary}

Now, we will prove Theorem \ref{rellich}.
\begin{proof}[Proof of Theorem \ref{rellich}] As in the proof of Theorem \ref{teo-hardy-carron}, for convenience, we will prove this theorem by considering the radial vector field $X=r\bar\na r$. By Lemma \ref{formula}, 
\begin{align*}
\frac{\De\psi^p}{(2-\ga)|X|^{\ga-2}} &= \frac{1}{2-\ga}\Big[\psi^p \De(\frac{1}{|X|^{\ga-2}}) + \dv_M\big[\frac{\na^M\psi^p}{|X|^{\ga-2}} - \psi^p \na^M(\frac{1}{|X|^{\ga-2}})\big]\Big]
\\& \ge \frac{1}{2-\ga}\dv_M Y
+\,\psi^p\,\big[\frac{k-\ga}{|X|^\ga}+\ga\frac{|X^\perp|^2}{|X|^{\ga+2}}+ \lan \frac{X}{|X|^\ga}, H\ran\big],
\end{align*}
where $Y = \frac{\na^M\psi^p}{|X|^{\ga-2}} + \frac{(\ga-2)\psi^p X^T}{|X|^\ga}$.
Using that $\ga>2$ and $\De\psi^p = p\psi^{p-2}[\psi\De\psi + (p-1)|\na^M \psi|^2]$, we obtain
\begin{align} 
\frac{p\psi^{p-1}\De\psi}{|X|^{\ga-2}} + \frac{p(p-1)\psi^{p-2}|\na^M \psi|^2}{|X|^{\ga-2}} &\le \ \dv_M Y 
\nonumber\\&\hspace{-2cm}\ -(\ga-2)\,\frac{\psi^p}{|X|^{\ga}}\,\big[k-\ga +\frac{\ga|X^\perp|^2}{|X|^{2}}+\lan X,H\ran\big].\label{eq-main}
\end{align}
Furthermore, the equality in (\ref{eq-main}) also remains valid when $\ga=2$ or when the radial curvature $(K_{\rad})_{\xi}=0$. Henceforth, we will assume $\ga\ge 2$ or $(K_{\rad})_\xi = 0$. So, (\ref{eq-main}) holds.


We write $\ga = \al+\be+1$, hence $\ga-2=(\al-1) +\be$. Using Young Inequality with $\si>0$, we have 
\begin{equation*}
p\De\psi\frac{\psi^{p-1}}{|X|^{\ga-2}} \ge -\frac{1}{p\si^p}\frac{|p\De\psi|^p}{|X|^{(\al-1)p}} - \frac{\si^q}{q}\frac{\psi^{(p-1)q}}{|X|^{\be q}},    
\end{equation*}
where $q=\frac{p}{p-1}$. As in the proof of Theorem \ref{teo-hardy-carron}, we take $\be q = \ga$. We obtain $\be = (p-1)(\al+1)$ and $p(\al-1) = \ga-2p$. Thus, 
\begin{equation}\label{cauchy1}
\frac{p\psi^{p-1}\De\psi}{|X|^{\ga-2}} \ge -\frac{1}{p\si^p}\frac{|p\De\psi|^p}{|X|^{\ga-2p}} - \frac{\si^q}{q}\frac{\psi^p}{|X|^\ga}.
\end{equation}
Applying (\ref{cauchy1}) into (\ref{eq-main}), one has
\begin{align} 
\big[(\ga-2)(k-\ga) &- \frac{\si^q}{q}\big]\frac{\psi^p}{|X|^\ga}  + p(p-1)\psi^{p-2}\frac{|\na^M \psi|^2}{|X|^{\ga-2}} \nonumber\\ &\le \dv Y + \frac{p^{p-1}}{\si^p}\frac{|\De\psi|^p}{|X|^{\ga-2p}} - (\ga-2) \frac{\psi^p}{|X|^\ga}\big[\frac{\ga|X^\perp|^2}{|X|^2} + \lan X,H\ran\big].\label{eq-main2}
\end{align}

On the other hand, we write $p=a+b+1$. By Lemma \ref{formula},
\begin{align}
\dv_M(\frac{\psi^p X^\top}{|X|^\ga}) & \ge \psi^p[\frac{k-\ga}{|X|^\ga}+\frac{\ga |X^\perp|^2}{|X|^{\ga+2}}]  + \lan p \na^M \psi + \psi H, \frac{\psi^{p-1} X}{|X|^\ga}\ran\nonumber
\\
&= \psi^p[\frac{k-\ga}{|X|^\ga} + \frac{\ga |X^\perp|^2}{|X|^{\ga+2}}]  + \lan \frac{p\psi^a \na^M \psi}{|X|^\al}, \frac{\psi^b X}{|X|^{\be+1}} \ran + \frac{\psi^p}{|X|^\ga}\lan H,X\ran. \label{cauchy-young2}
\end{align}
By Cauchy-Schwarz Inequality and Young Inequality with $\ep>0$,
\begin{equation}\label{young-ep2}
\lan \frac{p\psi^a \na^M \psi}{|X|^\al}, \frac{\psi^b X}{|X|^{\be+1}} \ran \ge -\frac{p^2}{2\ep^2}\frac{\psi^{2a}|\na^M\psi|^2}{|X|^{2\al}} - \frac{\ep^2}{2}\frac{\psi^{2b}}{|X|^{2\be}}.
\end{equation}
We take $2a=p-2$, $2b=p$, $2\be = \ga$ and $2\al=\ga-2$. By (\ref{cauchy-young2}) and (\ref{young-ep2}),
\begin{align*}
\dv_M(\frac{\psi^p X^\top}{|X|^\ga}) \ge \frac{\psi^p}{|X|^\ga}[k-\ga + \frac{\ga |X^\perp|^2}{|X|^2}] &+ \frac{\psi^p}{|X|^\ga}\lan \vec{H},X\ran\\& -\frac{p^2}{2\ep^2}\frac{\psi^{p-2}|\na^M\psi|^2}{|X|^{\ga-2}} - \frac{\ep^2}{2}\frac{\psi^p}{|X|^\ga}.
\end{align*}
Thus,
\begin{align*}
\frac{p^2}{2}\frac{\psi^{p-2}|\na^M\psi|^2}{|X|^{\ga-2}} &\ge  \frac{h(\ep) \psi^p}{|X|^\ga} + \frac{\ep^2\psi^p}{|X|^\ga}\Phi  -\dv_M(\frac{\ep^2\psi^p X^\top}{|X|^\ga}),\label{sigma}
\end{align*}
where $\Phi = \frac{\ga|X^\perp|^2}{|X|^2}+ \lan X,H \ran$ and $h(\ep)=\ep^2(k-\ga -\frac{\ep^2}{2})$. 
Plotting this inequality into (\ref{eq-main2}), we obtain
\begin{equation}\label{rel_bef}
A\frac{\psi^p}{|X|^\ga} + B\frac{\psi^p}{|X|^\ga}\Phi \le \dv_M Z + \frac{|\De\psi|^p}{|X|^{\ga-2p}},
\end{equation}
where 
\begin{align*}
A = A(\si,\ep) &= \frac{\si^p}{p^{p-1}}\big[(\ga-2)(k-\ga) - \frac{\si^q}{q} + \frac{2(p-1)h(\ep)}{p}\big];\\
B = B(\si,\ep) &=\frac{\si^p}{p^{p-1}} \big[\ga-2 + \frac{2(p-1)\ep^2}{p}\big];\\
Z = Z(\si,\ep) &= \frac{\si^p}{p^{p-1}} \big[Y + \frac{2(p-1)\ep^2}{p} \frac{\psi^p X^T}{|X|^\ga}\big] \\& = \frac{\si^p}{p^{p-1}}\big[\frac{\na^M\psi^p}{|X|^{\ga-2}} +  \big(\ga-2 +\frac{2(p-1)\ep^2}{p}\big)\frac{\psi^p X^T}{|X|^\ga}\big].
\end{align*}

In order to maximize $A(\si,\ep)$, note that $\frac{\p A}{\p\ep}=0$ if and only if $h'(\ep)=0$, which gives $\ep^2=k-\ga$, with $h(\ep)=\frac{(k-\ga)^2}{2}$. Hence, 
\begin{equation*}
A = \frac{\si^p}{p^{p-1}}[(\ga-2)(k-\ga) + \frac{p-1}{p}(k-\ga)^2 - \frac{\si^q}{q}] = \frac{\si^p}{p^{p-1}}[E - \frac{\si^q}{q}],
\end{equation*}
where 
\begin{equation}
E=(k-\ga)[\ga-2 + \frac{p-1}{p}(k-\ga)] = \frac{(k-\ga)}{p}[\ga-2 + (p-1)(k-2)].  
\end{equation}
Thus, $\frac{\p A}{\p \si} = 0$ if and only if $\si^q = E$. Note that if $k-\ga>0$ then $E>0$ if and only if $\ga>2-(k-2)(p-1)$.
So, by  hipothesis, $E>0$. We obtain $A = \frac{\si^p}{p^{p-1}}\frac{ E}{p}=\frac{\si^{p+q}}{p^p} = \frac{\si^{pq}}{p^p} = \frac{E^p}{p^p}$. It is easy to see that $\displaystyle\max_{\si,\ep} A(\si,\ep) = E^p/p$. Hence
\begin{equation*}
A = \frac{E^p}{p^p} = \frac{(k-\ga)^p}{p^{2p}}\big[\ga-2 + (p-1)(k-2)\big]^p.
\end{equation*}
Furthermore, 
\begin{align*}
B &= \frac{E^{p-1}}{p^{p-1}}\big[\ga-2 + \frac{2(p-1)(k-\ga)}{p}\big] = \frac{E^{p-1}}{p^{p-1}}\big[\frac{E}{k-\ga} + \frac{p-1}{p}\big]>0;\\
Z &= \frac{E^{p-1}}{p^{p-1}}\big[\frac{p\psi^{p-1}\na^M\psi}{|X|^{\ga-2}} +  \big[\ga-2 +\frac{2(p-1)(k-\ga)}{p}\big]\frac{\psi^p X^T}{|X|^\ga}\big]\\
&= \frac{\psi^{p-1}}{|X|^{\ga-2}}\big[\frac{ E^{p-1}}{p^{p-2}}\na^M\psi + B\frac{\psi X^T}{|X|^2}\big] = \frac{\psi^{p-1}}{|X|^{\ga-2}}W,
\end{align*}
where $W =\frac{E^{p-1}}{p^{p-2}}\na^M\psi + B\frac{\psi X^T}{|X|^2}$.

As in the proof of Theorem \ref{teo-hardy-carron}, for a small regular value $r_0$ of the geodesic distance $r=|X|$ satisfies 
\begin{align*}
\int_{M\cap [|X|>r_0]}\dv_MZ &=\int_{[\p M \cap [|X|>r_0]]\cup [M\cap [|X|=r_0]]}\lan Z,\nu\ran\nonumber\\
&=\int_{[\p M \cap [|X|>r_0]]}\lan Z,\nu\ran + \big[O(r_0^{2-\ga})+O(r_0^{1-\ga})\big]O(r_0^{k-1})\\
&=\int_{[\p M \cap [|X|>r_0]]}\lan Z,\nu\ran + O(r_0^{k-\ga})
\end{align*}
So, one has $\int_{M}\dv_MZ = \lim_{r_0\to 0}\int_{M\cap [|X|>r_0]}\dv_MZ = \int_{\p M} \lan Z,\nu\ran$.

Therefore, using (\ref{rel_bef}), 
\begin{equation*}
A\int_M \frac{\psi^p}{|X|^\ga} + B\int_M \frac{\psi^p}{|X|^\ga}\Phi^+ \\
\le \int_M \frac{|\De\psi|^p}{|X|^{\ga-2p}} + B\int_M \frac{\psi^p}{|X|^\ga}\Phi^- + \int_{\p M} \lan Z,\nu\ran.
\end{equation*}

Furthermore, if the equality holds, then $(K_{\rad})_{\xi}=0$ in $\spp\psi$ and the equality holds in (\ref{young-ep2}) with $\ep^2 = k-\ga$. Thus, it holds that $\na^M\psi = -\frac{k-\ga}{p}\psi X$. Following the same arguments in the proof of Theorem \ref{teo-hardy-carron} (after equation (\ref{radial})) we will obtain $\psi=0$. Theorem \ref{rellich} is proved.
\end{proof}

\section{Equality cases and blowup analysis}\label{sec-eq-blw}

Let us do a blowup analysis near the point $\xi$ in order to show that the integrals in (\ref{des-hardy}) converge. Note que  if $\xi\not\in M$ or $\ga\le 0$, then $\frac{1}{r^{\ga}}\in C^1(M)$ hence all the integrals in (\ref{des-hardy}) and (\ref{ineq-rellich}) converge. Assume $\ga>0$ and $\xi\in M$. Since $M$ is compact and isometrically immersed in $\bar M$, for a small  $r_0>0$, the intersection $M\cap [r\le r_0]$ contains only a finite number of connected components, all of them containing $\xi$ at their images. Furthermore, we may assume $|\bar\na r^T|\ge a>0$ at $M\cap [r\le r_0]$. In particular, all $0<s<r_0$ are regular values for $r$.  
By the coarea formula,
\begin{align}
\int_{M\cap [r\le r_0]} \frac{1}{r^\ga} &= \frac{1}{a}\int_{M\cap [r\le r_0]} \frac{1}{r^\ga}|\na r^T| = \frac{1}{a}\int_0^{r_0} \frac{1}{s^{\ga}}\int_{M\cap [r=s]} d\si^2 ds \nonumber\\& = \int_0^{r_0} \frac{1}{s^\ga} O(s^{k-1}) ds= O(1)\int_0^{r_0} s^{k-1-\ga}ds = O(r_0^{k-\ga}).\label{conv-int-M}
\end{align}
The last term follows from the fact that $k-\ga>0$. By similar arguments, since $\p M$ is also an isometrically immersed submanifold in $\bar M$, 
\begin{equation}\label{conv-int-bry}
\int_{\p M\cap [r\le r_0]} \frac{1}{r^{\ga-1}} = O(r_0^{k-\ga}).
\end{equation}

In this section, we will analyse functions that satisfy the equalities in Theorems \ref{teo-hardy-carron} and \ref{rellich}, for the case $p=1$. We recall that, for both theorems, a function that satisfies the equality for $p>1$ must be identically null. For $p=1$, the same conclusion fails, however, interesting conclusions are derived.

Let $B_R\subset \real^n$ be a ball of radius $R$ and center at some point $\xi \in \real^n$ and let $r$ be the distance function from $\xi$. We will prove the following

\begin{proposition}  Let $M^k$ be a compact submanifold with boundary $\p M$ immersed in $\real^n$. Assume $\p M$ is contained in the hypersphere $\p B_R = r^{-1}(R)$. For $0\le \psi\in C^1(M)$, the equality in Theorem \ref{teo-hardy-carron}, with $p=1$ and $-\infty<\ga<k$, occurs if and only if $\psi = \psi(r)$, i.e., it depends only on $r$, and $\psi'(r)\le 0$.
\end{proposition}

\begin{proof}
For all $0\le \psi\in C^1(M)$, it holds
\begin{equation}\label{hardy-eq-p=1}
(k-\ga)\int_M \frac{\psi}{r^\ga} + \int_M \frac{\psi}{r^\ga}\Phi
\le \int_M \frac{|\na^M \psi|}{r^{\ga-1}} + \int_{\p M}\frac{\psi}{r^{\ga-1}}\lan \bar\na r,\nu\ran,
\end{equation}
where $\Phi = \ga|\bar \na r^\perp|^2 + r\lan \bar\na r,H\ran$.

As in (\ref{eq-rel-radius}),
\begin{equation*}
\int_{\p M} \lan \psi\frac{\bar\na r^T}{r^{\ga-1}},\nu\ran =  \int_M \frac{\psi^p}{r^\ga}[k-\ga + \Phi]  + \lan \na^M \psi , \frac{\bar\na r}{r^{\ga-1}}\ran. 
\end{equation*}
We obtain
\begin{equation*}
\int_M \frac{|\na^M \psi|}{r^{\ga-1}} + \lan \na^M \psi , \frac{\bar\na r}{r^{\ga-1}}\ran \ge 0.
\end{equation*}
Thus, the equality in (\ref{hardy-eq-p=1}) holds if and only if $\na^M \psi = -|\na^M\psi|\bar\na r$, hence $\psi = \psi(r)$, i.e., depends only on $r$, and $\psi'(r)\le 0$.
\end{proof}

Now, we will analyse Theorem \ref{rellich} for the case $p=1$.

\begin{proposition} Let $M^k$ be a compact submanifold with boundary $\p M$ immersed in $\real^n$. Assume $\p M$ is contained in the hypersphere $\p B_R = r^{-1}(R)$. For $0\le \psi\in C^1(M)$, the equality in Theorem \ref{rellich}, with $p=1$ and $2<\ga<k$, occurs  if and only if $\De\psi\le 0$.
\end{proposition}
\begin{proof} Let $0\le \psi\in C^1(M)$. By Theorem \ref{rellich} with $p=1$ and $2<\ga< k$, 
 \begin{align}
(\ga-2)(k-\ga)\int_{M} \frac{\psi}{r^\ga} &+ (\ga-2)\int_{M} \frac{\psi}{r^\ga}\Phi \nonumber\\&\le \int_{M} \frac{|\De\psi|}{r^{\ga-2}} + \int_{\p M} \frac{1}{r^{\ga-2}}\lan \na^M\psi+(\ga-2)\psi\frac{\bar\na r^T}{r},\nu\ran.\label{ineq-prop-p=1}
\end{align}
Since $r=R$ in $\p M$, 
\begin{equation}\label{eq-rel-De}
\int_{\p M} \frac{1}{r^{\ga-2}}\lan \na^M\psi,\nu \ran = \frac{1}{R^{\ga-2}}\int_{\p M} \lan \na^M\psi,\nu \ran = \frac{1}{R^{\ga-2}}\int_{M} \De\psi,
\end{equation}
and, by Lemma \ref{formula} (using that $X=r\bar\na r$), 
\begin{align}\label{eq-rel-radius}
\int_{\p M} \lan \psi\frac{\bar\na r^T}{r^{\ga-1}},\nu\ran &= \int_{\p M} \lan \psi\frac{X^T}{|X|^\ga},\nu\ran = \int_M \dv_M(\psi\frac{X^T}{|X|^\ga}) \\ &= 
\int_M \frac{\psi}{|X|^\ga}[k-\ga+\frac{\ga |X^\perp|^2}{|X|^2}]  + \lan \na^M \psi + \psi H, \frac{X}{|X|^\ga}\ran \nonumber\\&=
\int_M \frac{\psi^p}{r^\ga}[k-\ga + \Phi]  + \lan \na^M \psi , \frac{\bar\na r}{r^{\ga-1}}\ran \nonumber.
\end{align}
Furthermore,
\begin{align}\label{eq-rel-fin}
\int_M \lan \na^M \psi , \frac{\bar\na r}{r^{\ga-1}}\ran &= -\frac{1}{\ga-2}\int_M \lan \na^M \psi , \bar\na r^{2-\ga}\ran  
\\&= -\frac{1}{\ga-2}\int_M [\dv(r^{2-\ga}\na^M\psi) - r^{2-\ga}\De\psi] 
\nonumber\\&=
-\frac{R^{2-\ga}}{\ga-2}\int_{\p M}\lan \na^M\psi,\nu\ran + \frac{1}{\ga-2}\int_{M} r^{2-\ga}\De\psi \nonumber
\nonumber\\&=
-\frac{R^{2-\ga}}{\ga-2}\int_{M} \De\psi  + \frac{1}{\ga-2}\int_{M} r^{2-\ga}\De\psi\nonumber. 
\end{align}
By (\ref{eq-rel-De}), (\ref{eq-rel-radius}) and (\ref{eq-rel-fin}), we obtain
\begin{equation*}
0 \le \int_M \frac{1}{r^{\ga-2}}(|\De\psi| + \De\psi) = \int_M \frac{2}{r^{\ga-2}}(\De\psi)^+   
\end{equation*}
Thus, using that $|\De\psi|=(\De\psi)^+ + (\De\psi)^-$ and $\De\psi=(\De\psi)^+ - (\De\psi)^-$, the equality in (\ref{ineq-prop-p=1}) holds if and only if $\De\psi\le 0$ in $M$. 
\end{proof}

\section{Applications}
We recall that an isometric immersion $f:M\to \real^n$ is said to be self-shrinker if the mean curvature vector satisfies \begin{equation*}
H=-\frac{1}{2} f^\perp.
\end{equation*} 
In this case, the radial vector $X=r\bar\na r$ coincides with $f$, in particular the support function $\lan H,\bar\na r\ran =-\frac{r}{2}|\bar\na r^\perp|^2$. By Theorem \ref{teo-hardy}, it follows 
\begin{proposition}\label{cor-self-shinker} Let $M^k$ be a closed manifold isometrically immersed in a Hadamard manifold $\bar M$. Let $r = d_{\bar M}(\cdot\,,\xi)$ be the distance in $\bar M$ from any fixed point $\xi$. Assume the mean curvature of $M$ satisfies 
\begin{equation}\label{cond-self-shrinker}
\lan H,\bar\na r\ran \ge -\frac{\la r}{2}|\bar\na r^\perp|^2,
\end{equation}
for some $\la>0$. Then, $\sup_M r^2\ge 2k/\la$.
\end{proposition}
\begin{proof} By contradiction, we assume $\sup_M r^2<2k/\la$ and choose $\ga<k$ so that $\sup_M r^2<2\ga/\la$.  
Thus, 
\begin{equation*}
\Phi  = \ga|\bar \na r^\perp|^2 + r\lan \bar\na r,H\ran \ge (\ga - \frac{\la r^2}{2})|\bar \na r^\perp|^2 > 0.
\end{equation*}
So, $\Phi^-=0$. Applying Theorem \ref{teo-hardy} with $\psi=1$, we obtain a contradiction. 
\end{proof}
We remark that Proposition \ref{cor-self-shinker} is optimal. 
In fact, the round $k$-dimensional hypersphere $S^{k}(0,{\sqrt{2k}})\subset\real^{k+1}$ of radius $\sqrt{2k}$ and centered at the origin is self-shrinker. So, the equality in (\ref{cond-self-shrinker}) holds with $\la=1$ and $r^2=2k$.
\\

We define the {\it extrinsic diameter} $D(M)$ of a submanifold $M$ in  $\bar M$ as the infimum of the diameter among all geodesic ball in $\bar M$ that contains $M$.

\begin{theorem}\label{est-first-eigenv} Let $M^k$ be a  compact manifold with boundary $\p M\neq \emptyset$ isometrically immersed in a Hadamard manifold $\bar M$ and let $D = D(M)$ be the extrinsic diameter of $M$. Then, the first eigenvalue of the Laplacian of $M$ satisfies
\begin{equation*}
\la_1(M) \ge \frac{k^2}{D^2}(1 - \frac{D}{k}\sup_M|H|).
\end{equation*}
\end{theorem}
\begin{proof} Let $0\le \psi \in C^1(M)$, with $\psi\not\equiv 0$ and $\psi=0$ on $\p M$. Using Theorem \ref{teo-hardy}, with $\ga=0$ and $p=2$, we obtain 
\begin{equation*}
\frac{k^2}{4}\int_M \psi^2 + \frac{k}{2}\int_M \psi^2\Phi \le \int_M r^2|\na^M \psi|^2 \le \frac{D^2}{4}\int_M |\na^M \psi|^2. 
\end{equation*}
Note that, in this case, $\Phi = r\lan\bar \na r, H\ran \ge -\frac{D}{2}\sup_M|H|$. Hence,
\begin{equation*}
\frac{k^2}{4}\int_M \psi^2 - \frac{kD}{4}\sup_M|H| \int_M \psi^2 \le \int_M r^2|\na^M \psi|^2 \le \frac{D^2}{4}\int_M |\na^M \psi|^2.
\end{equation*}
Since $\la_1(M)=\inf_{\mathcal F} \frac{\int_M |\na^M\psi|^2}{\int_M \psi^2}$, where $\mathcal F$ is the set of all $0\le \psi\in C^1(M)$, with $\psi\not\equiv 0$ and $\psi=0$ on $\p M$. We obtain
\begin{equation*}
\frac{k^2}{4} - \frac{kD}{4}\sup_M|H| \le \frac{D^2}{4}\la_1(M).
\end{equation*}
Theorem \ref{est-first-eigenv} follows. 
\end{proof}


\end{document}